\documentclass[11pt]{amsart}

\newtheorem{theorem}{Theorem}[section]

\newtheorem{corollary}[theorem]{Corollary}

\newtheorem{lemma}[theorem]{Lemma}

\newtheorem{proposition}[theorem]{Proposition}

\theoremstyle{definition}

\usepackage[colorlinks, bookmarks=true]{hyperref}
\usepackage{color,graphicx,shortvrb}

\usepackage{enumerate}
\usepackage{amsmath}
\usepackage{amssymb}

\usepackage[latin 1]{inputenc}

\def\11{\textbf{$1$}}

\newcommand{\ran}{\mathrm{ran} \,}

\begin{document}

\numberwithin{equation}{section}

\title[2-local triple derivations]{2-local triple derivations on von Neumann algebras}

\author[Kudaybergenov]{Karimbergen Kudaybergenov}
\address{Ch. Abdirov 1, Department of Mathematics, Karakalpak State University, Nukus 230113, Uzbekistan}
\email{karim2006@mail.ru}

\author[Oikhberg]{Timur Oikhberg}
\address{Department of Mathematics, University of Illinois, 
Urbana IL 61801, USA}
\email{oikhberg@illinois.edu}

\author[Peralta]{Antonio M. Peralta}
\address{Departamento de An{\'a}lisis Matem{\'a}tico, Facultad de
Ciencias, Universidad de Granada, 18071 Granada, Spain.}
\email{aperalta@ugr.es}
\curraddr{Visiting Professor at Department of Mathematics, College of Science, King Saud University, P.O.Box 2455-5, Riyadh-11451, Kingdom of Saudi Arabia.}

\author[Russo]{Bernard Russo}
\address{Department of Mathematics, University of California, Irvine CA 92697-3875, USA}
\email{brusso@uci.edu}

\thanks{Second author partially supported by Simons Foundation travel grant 210060.
Third author partially supported by the Spanish Ministry of Science and Innovation,
D.G.I. project no. MTM2011-23843, Junta de Andaluc\'{\i}a grant FQM375, and the Deanship of Scientific Research
at King Saud University (Saudi Arabia) research group no. RGP-361.}



\date{}
\maketitle

\begin{abstract} We prove that every {\rm(}not necessarily linear nor continuous{\rm)} 2-local triple derivation on a von Neumann algebra $M$ is a triple derivation, equivalently, the set Der$_{t} (M)$, of all triple derivations on $M,$ is algebraically 2-reflexive in the set $\mathcal{M}(M)= M^M$ of all mappings from $M$ into $M$.
\end{abstract}

\maketitle
\thispagestyle{empty}

\section{Introduction}\label{sec:intro}

Let $X$ and $Y$ be Banach spaces. According to the terminology employed in the literature (see, for example, \cite{BattMol}), a subset $\mathcal{D}$ of the Banach space $B(X,Y)$, of all bounded linear operators from $X$ into $Y$, is called \emph{algebraically reflexive} in $B(X,Y)$ when it satisfies the property:
\begin{equation}\label{eq reflexivity} T\in B(X,Y) \hbox{ with } T(x)\in \mathcal{D} (x), \ \forall x\in X \Rightarrow T\in \mathcal{D}.
\end{equation}
\emph{Algebraic reflexivity} of ${\mathcal D}$ in the space $L(X,Y)$, of all linear mappings from $X$ into $Y$, a stronger version of the above property not requiring continuity of $T$, is defined by:
\begin{equation}\label{eq reflexivity2} T\in L(X,Y) \hbox{ with } T(x)\in \mathcal{D} (x), \ \forall x\in X \Rightarrow T\in \mathcal{D}.
\end{equation}

In 1990, Kadison proved that (\ref{eq reflexivity}) holds if ${\mathcal D}$ is the set Der$(M,X)$ of all (associative) derivations on a von Neumann algebra $M$ into a dual $M$-bimodule $X$ \cite{Kad90}.   Johnson extended Kadison's result by establishing that the set ${\mathcal D}=$ Der$(A,X),$ of all (associative) derivations from a C$^*$-algebra $A$ into a Banach $A$-bimodule $X$ satisfies (\ref{eq reflexivity2}) \cite{John01}.

 Algebraic reflexivity of the set of local triple derivations on a C$^*$-algebra and on a JB$^*$-triple have been studied in \cite{Mack, BurFerGarPe2014, BurFerPe2013} and \cite{FerMolPe}. More precisely, Mackey proves in \cite{Mack} that the set ${\mathcal D}=\hbox{Der}_{t} (M),$ of all triple derivations on a JBW$^*$-triple $M$ satisfies (\ref{eq reflexivity}). The result has been supplemented in \cite{BurFerPe2013}, where Burgos, Fernández-Polo and the third author of this note prove that for each JB$^*$-triple $E$, the set ${\mathcal D}=\hbox{Der}_{t} (E)$ of all triple derivations on $E$ satisfies (\ref{eq reflexivity2}). \smallskip

Hereafter, {\it algebraic reflexivity} will refer to the stronger version  (\ref{eq reflexivity2}) which does not assume the continuity of $T$.\smallskip

In \cite{BreSemrl95}, Bre\v{s}ar and \v{S}emrl proved that the set of all (algebra) automorphisms of $B(H)$ is algebraically reflexive  whenever $H$ is a separable, infinite-dimensional Hilbert space. Given a Banach space $X$. A linear mapping $T: X\to X$ satisfying the hypothesis at \eqref{eq reflexivity2} for $\mathcal{D} = \hbox{Aut} (X)$, the set of automorphisms on $X$, is called a \emph{local automorphism}.
Larson and Sourour showed in \cite{LarSou} that for every infinite dimensional Banach space $X$, every surjective local automorphism $T$ on the Banach algebra $B(X),$  of all bounded linear operators on $X$, is an automorphism.
\smallskip

Motivated by the results of \v{S}emrl in \cite{Semrl97}, references witness a growing interest in a subtle version of algebraic reflexivity called \emph{algebraic 2-reflexivity} (cf. \cite{AyuKuday2012,AyuKuday2014, BurFerGarPe2014preprint, BurFerGarPe2014preprintTripleHom,KimKim05,LiuWong06,Mol2002,Mol2003} and \cite{Pop}). A subset $\mathcal{D}$ of the set $\mathcal{M}(X,Y) = Y^{X},$ of all mappings from $X$ into $Y$, is called \emph{algebraically 2-reflexive} when the following property holds: for each mapping $T$ in $\mathcal{M}(X,Y)$ such that for each $a,b\in X,$ there exists $S= S_{a,b}\in  \mathcal{D}$ (depending on $a$ and $b$), with $T(a) =S_{a,b} (a)$ and $T(b) = S_{a,b} (b)$, then $T$ lies in $\mathcal{D}.$ A mapping $T: X\to Y$ satisfying that for each $a,b\in X,$ there exists $S= S_{a,b}\in  \mathcal{D}$ (depending on $a$ and $b$), with $T(a) =S_{a,b} (a)$ and $T(b) = S_{a,b} (b)$ will be called a 2-local $\mathcal{D}$-mapping. If we assume that every mapping $s\in \mathcal{D}$ is $r$-
homogeneous (that is, $S(t a) = t^r S(a)$ for every $t\in \mathbb{R}$ or $\mathbb{C}$) with $0<r$, then every 2-local $\mathcal{D}$-mapping $T: X\to Y$ is $r$-homogeneous. Indeed, for each $a\in X$, $t\in \mathbb{C}$ take $S_{a,ta}\in \mathcal{D}$ satisfying $T(t a) = S_{a,ta} (ta) = t^r S_{a,ta} (a) = t^r T(a)$.\smallskip

\v{S}emrl establishes in \cite{Semrl97} that for every infinite-dimensional separable Hilbert space $H$, the sets Aut$(B(H))$ and Der$(B(H))$, of all  (algebra) automorphisms and associative derivations on $B(H)$, respectively, are algebraically 2-reflexive in $\mathcal{M}(B(H)) = \mathcal{M}(B(H),B(H)).$ Ayupov and the first author of this note proved in \cite{AyuKuday2012} that the same statement remains true for general Hilbert spaces (see \cite{KimKim04} for the finite dimensional case). Actually, the set Hom$(A)$, of all homomorphisms on a general C$^*$-algebra $A,$ is algebraically 2-reflexive in the Banach algebra $B(A)$, of all bounded linear operators on $A$, and the set $^*$-Hom$(A)$, of all $^*$-homomorphisms on $A,$ is algebraically 2-reflexive in the space $L(A)$, of all linear operators on $A$ (cf. \cite{Pe2014}).\smallskip

In recent contributions,  Burgos, Fernández-Polo and the third author of this note prove that the set Hom$(M)$ (respectively, Hom$_{t} (M)$), of all homomorphisms (respectively, triple homomorphisms) on a von Neumann algebra (respectively, on a JBW$^*$-triple) $M$, is an algebraically 2-reflexive subset of $\mathcal{M}(M)$ (cf. \cite{BurFerGarPe2014preprint}, \cite{BurFerGarPe2014preprintTripleHom}, respectively), while Ayupov and the first author of this note establish that set Der$(M)$ of all derivations on $M$ is algebraically 2-reflexive in $\mathcal{M}(M)$ (see \cite{AyuKuday2014}).\smallskip

In this paper, we consider the set Der$_{t} (A)$ of all triple derivations on a C$^*$-algebra $A$. We recall that every C$^*$-algebra $A$ can be equipped with a ternary product of the form $$\{a,b,c\} = \frac12 (a b^* c + c b^* a). $$ When $A$ is equipped with this product it becomes a JB$^*$-triple in the sense of \cite{Ka83}. A linear mapping $\delta: A\to A$ is said to be a \emph{triple derivation} when it satisfies the (triple) Leibnitz rule: $$\delta\{a,b,c\} = \{\delta(a),b,c\} + \{a,\delta(b),c\}+ \{a,b,\delta(c)\}.$$ It is known that every triple derivation is automatically continuous (cf. \cite{BarFri}). We refer to \cite{BarFri,HoMarPeRu} and \cite{PeRu} for the basic references on triple derivations. According to the standard notation, 2-local Der$_{t} (A)$-mappings from $A$ into $A$ are called \emph{2-local triple derivations}. \smallskip

The goal of this note is to explore the algebraic 2-reflexivity of Der$_{t} (A)$ in  $\mathcal{M}(A)$. Our main result proves that every {\rm(}not necessarily linear nor continuous{\rm)} 2-local triple derivation on an arbitrary von Neumann algebra $M$ is a triple derivation {\rm(}hence linear and continuous{\rm)} (see Theorem \ref {mainthm}), equivalently, Der$_{t} (M)$ is algebraically 2-reflexive in $\mathcal{M}(M)$.

\section{2-local triple derivations on von Neumann algebras}
\label{s:vna}

We start by recalling some generalities on triple derivations. Let $A$ be a C$^*$-algebra. For each $b\in A,$ we shall denote by $M_b$ the
Jordan multiplication mapping by the element $b,$ that is $M_b (x)= b\circ x = \frac12 (b x+ x b).$ Following standard notation, given elements $a,b$ in $A$, we denote by $L(a,b)$ the operator on $A$ defined by $L(a,b) (x) = \{a,b,x\} = \frac12 ( ab^* x+  xb^*a)$. It is known that he mapping $\delta(a,b) : A\to A,$ given by $$\delta(a,b)(x)= L(a,b) (x) - L(b,a) (x),$$ is a triple derivation on $A$ (cf. \cite{BarFri,HoMarPeRu}), called an inner triple derivation.\smallskip

Let
$\delta: A \to A$ be a triple derivation on a unital
C$^\ast$-algebra. By \cite[Lemmas 1 and 2]{HoMarPeRu},
$\delta(\mathbf{1})^\ast =-\delta (\mathbf{1}),$ and
$M_{\delta(\mathbf{1})} = \delta (\frac12 \delta
(\mathbf{1}),\mathbf{1})$ is an inner triple derivation on $A$ and
the difference $D = \delta - \delta (\frac12 \delta
(\mathbf{1}),\mathbf{1})$ is a Jordan $^\ast$-derivation on $A,$
more concretely,
$$
D (x\circ y ) = D(x) \circ y + x \circ D(y), \hbox{ and }
D(x^\ast) =D(x)^\ast,
$$
for every $x,y\in A.$ By \cite[Corollary 2.2]{BarFri}, $\delta$
(and hence $D$) is a continuous operator. A widely known result,
due to B.E. Johnson, states that every bounded Jordan derivation
from a C$^\ast$-algebra $A$ to a Banach $A$-bimodule is an
associative derivation (cf. \cite{John96}). Therefore, $D$ is an
associative $^\ast$-derivation in the usual sense. When $A=M$ is a
von Neumann algebra, we can guarantee that $D$ is an inner
derivation, that is there exists $\widetilde{a}\in A$ satisfying
$D(x) = [\widetilde{a},x] = \widetilde{a} x- x \widetilde{a},$ for
every $x\in A$ (cf. \cite[Theorem 4.1.6]{Sak}). Further, from the
condition $D(x^\ast ) = D(x)^\ast,$ for every $x\in A,$ we deduce
that $(\widetilde{a}^\ast +\widetilde{a}) x = x
(\widetilde{a}^\ast +\widetilde{a}).$ Thus, taking
$a=\displaystyle\frac{1}{2} (\widetilde{a}-\widetilde{a}^\ast),$
it follows that $[a,x] =[\widetilde{a},x],$ for every $x\in M.$ We
have therefore shown that for every triple derivation $\delta$ on
a von Neumann algebra $M,$ there exist skew-hermitian elements
$a,b\in M$ satisfying
$$
\delta (x) = [a,x] + b\circ x,
$$ for every $x\in M.$\smallskip

 Our first lemma is a direct consequence of the above arguments (see \cite[Lemmas 1 and 2]{HoMarPeRu}).

\begin{lemma}\label{l 2-local T(1) skew}
Let $T: A\to A$ be a {\rm(}not necessarily linear nor
continuous{\rm)} 2-local triple derivation on a unital
C$^*$-algebra. Then
\begin{enumerate}[$(a)$]
\item $T(\mathbf{1})^\ast = -T(\mathbf{1});$
\item  $M_{T(\mathbf{1})} =
\delta \left(\frac12 T (\mathbf{1}),\mathbf{1}\right)$ is an inner
triple derivation on $A;$
\item   $\widehat{T}=T - \delta \left(\frac12 T
(\mathbf{1}),\mathbf{1}\right)$ is a 2-local triple derivation on
$A$ with $\widehat{T}(\mathbf{1}) = 0.$
\end{enumerate} $\hfill\Box$
\end{lemma}

In what follows, we denote by $A_{sa}$ the hermitian elements of the C$^*$-algebra $A$.
\begin{lemma}\label{l sead} Let $T: A\to A$ be a {\rm(}not necessarily linear nor continuous{\rm)}
2-local triple derivation on a unital C$^\ast$-algebra satisfying
$T(\mathbf{1})=0.$ Then $T(x)=T(x)^\ast$ for all $x\in A_{sa}.$
\end{lemma}

\begin{proof} Let $x\in A_{sa}.$ By assumptions,
$$
T(x)^\ast = \{\mathbf{1},T(x),\mathbf{1}\} =
\{\mathbf{1},\delta_{x,\mathbf{1}}(x),\mathbf{1}\} =
\delta_{x,\mathbf{1}} \{\mathbf{1},x,\mathbf{1}\} - 2
\{\delta_{x,\mathbf{1}}(\mathbf{1}),x,\mathbf{1}\}
$$
$$
= \delta_{x,\mathbf{1}} (x^\ast) - 2
\{T(\mathbf{1}),x,\mathbf{1}\} = \delta_{x,\mathbf{1}} (x) = T(x).
$$
The proof is complete.
\end{proof}

\begin{lemma}\label{l sea} Let $T: M\to M$ be a {\rm(}not necessarily linear nor continuous{\rm)}
2-local triple derivation on a von Neumann algebra satisfying
$T(\mathbf{1})=0.$ Then for every $x, y\in M_{sa}$ there exists a
skew-hermitian element $a_{x,y}\in M$ such that
$$
T (x)=[a_{x,y},x],\hbox{ and, } T(y)=[a_{x,y},y].
$$
\end{lemma}

\begin{proof}
For every $x, y\in M_{sa}$ we can find skew-hermitian elements $a_{x,y}, b_{x,y}\in
M$ such that
$$
T(x)=[a_{x,y},x]+b_{x,y}\circ x,\hbox{ and, }T(y)=[a_{x,y},y]+b_{x,y}\circ y.
$$
Taking into account that $T(x)=T(x)^\ast$ (see Lemma \ref{l sead}) we obtain
$$
[a_{x,y},x]+b_{x,y}\circ x=T(x)=T(x)^\ast=[a_{x,y},
x]^\ast+(b_{x,y}\circ x)^\ast
$$
$$
=[x, a_{x,y}^\ast]+x\circ b_{x,y}^\ast=[x, -a_{x,y}]-x\circ
b_{x,y}=[a_{x,y},x]-b_{x,y}\circ x,
$$
i.e. $b_{x,y}\circ x=0,$ and similarly $b_{x,y}\circ y=0.$ Therefore $T(x)=[a_{x,y},x],$ $T(y)=[a_{x,y},y],$ and the proof is complete.
\end{proof}

We state now an observation, which plays
an useful role in our study.

\begin{lemma}\label{l:ax=xa, bx=xb=0} Let $a$ and $b$ be skew-hermitian elements in a C$^\ast$-algebra  $A$.
Suppose $x \in A$ is self-adjoint with $ax - xa + bx + xb = 0.$
Then $ax = xa,$ and $bx = - xb.$
\end{lemma}

\begin{proof}
Since $0 = ax - xa + bx + xb$. Passing to the adjoint, we obtain
$ax - xa - (bx + xb) = 0$. Conclude the proof by adding and
subtracting these two equalities. The proof is complete.
\end{proof}

Let $M$ be a von Neumann algebra. If $x \in M_{sa}$, we denote by $s(x)$ the support projection of $x$ --
that is, the projection onto $(\ker (x))^\perp = \overline{\ran (x)}$.
We say that $x$ \emph{has full support} if $s(x) = 1$ (equivalently,
$\ker (x) = \{0\}$).

\begin{lemma}\label{l:product} Let $M$ be a von Neumann algebra.
Suppose $u \in M_+$ has full support, $c \in M$ is self-adjoint,
and $\sigma(c^2 u) \cap (0,\infty) = \emptyset$. Then $c = 0.$
Consequently, if $u$ and $c$ are as above, and $uc + cu = 0$
{\rm(}or $c^2 u = - cuc \leq 0${\rm)}, then $c = 0.$
\end{lemma}

\begin{proof} For the fist statement of the lemma,
suppose $\sigma(c^2 u) \cap (0,\infty) = \emptyset$. Note that
$$(-\infty,0]\supseteq
\sigma(c^2 u) \cup \{0\} = \sigma (c \cdot c u) \supseteq \sigma(c
u c).
$$
However, $cuc$ is positive, hence $\sigma(cuc) \subset [0,
\|cuc\|]$, with $\max_{\lambda \in \sigma(cuc)} = \|cuc\|.$ Thus,
$c u^{1/2} u^{1/2} c = cuc = 0,$ which means that $c u^{1/2} =
u^{1/2} c=0$ and hence $s(c) \subset 1-(u^{1/2}) = 1-s(u)=0,$
which leads to $c = 0.$\smallskip

To prove the second part, we have $c^2 u = - cuc \leq 0,$ hence in
particular, $\sigma(c^2 u) \subset (-\infty,0].$ The proof is
complete. \end{proof}

In \cite[Lemma 2.2]{AyuKuday2014}, Ayupov and the first author of
this note prove that for every {\rm(}not necessarily linear nor
continuous{\rm)} $2$-local derivation on a von Neumann algebra
$\Delta:M\to M$, and every self-adjoint element $z\in M$, there
exists $a\in M$ satisfying $$\Delta (x) = [a,x],$$ for every $x\in
\mathcal{W}^\ast(z)$, where $\mathcal{W}^\ast(z)=\{z\}''$ denotes the
abelian von Neumann subalgebra of $M$ generated by the element $z$,
and the unit element and $\{z\}''$ denotes the bicommutant of the set $\{z\}$.
We prove next a ternary version of this
result.

\begin{lemma}\label{l linearity on single generated von Neumann subalgebras}
Let $T: M \to M$ be a {\rm(}not necessarily linear nor continuous{\rm)} $2$-local
triple derivation on a von Neumann algebra. Let $z\in M$ be a self-adjoint element and let
$\mathcal{W}^\ast(z)=\{z\}''$ be the abelian von Neumann subalgebra of $M$
generated by the element $z$ and the unit element. Then there exist skew-hermitian
elements $a_z, b_z\in M$, depending on $z$, such that
\begin{equation*}\label{spat}
T(x)=[a_z,x] + b_z\circ x = a_z x-x a_z + \frac12 (b_z x+x b_z)
\end{equation*}
for all $x\in \mathcal{W}^\ast(z).$  In particular, $T$ is linear on $\mathcal{W}^\ast(z).$
\end{lemma}

\begin{proof} We can assume that $z\neq 0$. Note that the abelian von Neumann subalgebras
generated by $\mathbf 1$ and $z$ and by $\mathbf 1$  and $\mathbf 1+\frac{\textstyle z}{\textstyle  2\|z\|}$
coincide. So, replacing $z$ with  $\textbf{1}+\frac{\textstyle z}{\textstyle  2\|z\|}$
we can assume that $z$ is an invertible positive element.\smallskip

By definition, there exist skew-hermitian elements $a_{z}, b_{z}\in M$
(depending on $z$) such that $$ T(z) = [a_{z},z]+ b_{z} \circ z.$$

Define a mapping $T_0 : M \to M$ given by $T_0 (x) = T (x)- ([a_{z},z]+ b_{z} \circ z),$
$\ x\in M.$ Clearly, $T_0$ is a 2-local triple derivation on $M$.
We shall show that $T_0\equiv 0$ on $\mathcal{W}^\ast(z)$. Let $x\in \mathcal{W}^\ast(z)$ be an
arbitrary element. By assumptions, there exist skew-hermitian
elements $c_{z,x}, d_{z,x}\in M$ such that
$$
T_0(z)=[c_{z,x}, z]+ d_{z,x}\circ z,\hbox{ and, }  T_0(x)=[c_{z,x}, x]+d_{z,x}\circ x.
$$
Since
$$
0=T_0(z)=[c_{z,x}, z]+d_{z,x}\circ z,
$$
we get
$$
[c_{z,x}, z]+d_{z,x}\circ z=0.
$$
Taking into account that  $z$ is a hermitian element and Lemma \ref{l:ax=xa, bx=xb=0} we get
$c_{z,x} z=z c_{z,x}$ and $d_{z,x} z =- z d_{z,x}.$

Since $z$ has a full support, and $d_{z,x}^2 z =- d_{z,x} z d_{z,x}$, Lemma \ref{l:product} implies that  $d_{z,x}=0.$
Further
$$
c_{z,x}\in \{z\}'=\{z\}'''=\mathcal{W}^\ast(z)',
$$
i.e. $c_{z,x}$ commutes with any element in $\mathcal{W}^\ast(z).$
Therefore
$$
T_0(x) = [c_{z,x}, x]+d_{z,x}\circ x=0
$$
for all $x\in \mathcal{W}^\ast(z).$ The proof is complete.
\end{proof}

\subsection{Complete additivity of 2-local derivations and 2-local triple derivations on von Neumann algebras}\ \newline

Let $\mathcal{P} (M)$ denote the lattice of projections in a von
Neumann algebra $M.$
Let $X$ be a Banach space. A mapping $\mu:
\mathcal{P} (M)\to X$ is said to be \emph{finitely additive} when
$$
\mu \left(\sum\limits_{i=1}^n p_i\right) = \sum\limits_{i=1}^{n}
\mu (p_i),
$$
for every family $p_1,\ldots, p_n$ of mutually orthogonal
projections in $M.$ A mapping $\mu: \mathcal{P} (M)\to X$ is said to be \emph{bounded} when the set
$$
\left\{ \|\mu (p)\|: p \in \mathcal{P} (M) \right\}
$$ is bounded.\smallskip

The celebrated Bunce-Wright-Mackey-Gleason theorem (\cite{BuWri92, BuWri94}) states that if $M$ has no summand of type $I_2$, then every bounded finitely additive mapping $\mu: \mathcal{P} (M)\to X$ extends to a bounded linear operator from $M$ to $X$.\smallskip

According to the terminology employed in \cite{Shers2008} and \cite{Doro}, a completely
additive mapping $\mu : \mathcal{P} (M)\to \mathbb{C}$ is called a \emph{charge}.
The  Dorofeev--Sherstnev
 theorem (\cite[Theorem 29.5]{Shers2008} or \cite[Theorem 2]{Doro})  states that any charge on a  von Neumann algebra with no summands of type $I_n$ is bounded.\smallskip

We shall use the Dorofeev-Shertsnev theorem in Corollary~\ref{c 2 local triple derivations are bounded charges} in order to be able  to apply the Bunce-Wright-Mackey-Gleason theorem in Proposition~\ref{p addit no type I_n}.   To this end, we need Proposition~\ref{p AyupovKudaybergenov sigma complete additivity ternary}, which is implicitly applied in \cite[proof of Lemma 2.3]{AyuKuday2014} for 2-local associative derivations.
A  proof is included here for completeness reasons.\smallskip


First, we recall some facts about the strong$^*$-topology.  For each normal positive functional $\varphi$ in the predual of a von Neumann algebra $M$, the mapping $$x\mapsto \|x\|_{\varphi} = \left(\varphi (\frac{x x^* + x^* x}{2}) \right)^{\frac12}\quad (x\in M)$$ defines a prehilbertian seminorm on $M$. The \emph{strong$^*$ topology} of $M$ is the locally convex topology on $M$ defined by all the seminorms $\|.\|_{\varphi}$, where $\varphi$ runs in the set of all positive functionals in $M_*$ (cf. \cite[Definition 1.8.7]{Sak}). It is known that the strong$^*$ topology of $M$ is compatible with the duality $(M,M_{*})$, that is a functional $\psi: M \to \mathbb{C}$ is strong$^*$ continuous if and only if it is weak$^*$ continuous (see \cite[Corollary 1.8.10]{Sak}).
We also recall that the product of every von Neumann algebra is jointly strong$^*$ continuous on bounded sets (see \cite[Proposition 1.8.12]{Sak}).\smallskip

Suppose $X=W$ is another von Neumann algebra,
and let $\tau$ denote the norm-, the weak$^\ast$- or the
strong$^\ast$-topology of $W.$ The mapping $\mu$ is said to be
\emph{$\tau$-completely additive} (respectively, \emph{countably
or sequentially $\tau$-additive}) when
\begin{equation}\label{eq completely additive}
\mu\left(\sum\limits_{i\in I} e_i\right) =\tau\hbox{-}
\sum\limits_{i\in I}\mu(e_i)
\end{equation} for every family (respectively, sequence)  $\{e_i\}_{i\in I}$ of mutually orthogonal
projections in $M.$

It is known that every family $(p_i)_{i \in I}$ of mutually orthogonal projections in
a von Neumann algebra $M$ is summable with respect to the weak$^*$-topology of  $M$
and $\displaystyle p = \hbox{weak}^*\hbox{-}\sum_{i\in I} p_i$ is a projection in $M$
(cf. \cite[Definition 1.13.4]{Sak}). Further, for each normal positive functional
$\phi$ in $M_*$ and every finite set $F\subset I,$ we have
$$\left\|p - \sum_{i\in F} p_i \right\|_{\phi}^2 = \phi \left( p - \sum_{i\in F} p_i \right),$$
which implies that the family $(p_i)_{i \in I}$ is summable with respect to the strong$^*$-topology of $M$
with the same limit, that is, $\displaystyle p = \hbox{strong}^*\hbox{-}\sum_{i\in I} p_i$ .\smallskip

\begin{proposition}\label{p AyupovKudaybergenov sigma complete additivity ternary}
Let $T: M\to M$ be a {\rm(}not necessarily linear nor continuous{\rm)} 2-local triple
derivation on a von Neumann algebra. Then the following statements hold:\begin{enumerate}[$(a)$]
\item The restriction $T|_{\mathcal{P}(M)}$ is sequentially strong$^*$-additive, and consequently sequentially weak$^*$-additive;
\item $T|_{\mathcal{P}(M)}$ is weak$^*$-completely additive, i.e.,
\begin{equation}\label{ca}
T\left(\hbox{weak$^*$-}\sum\limits_{i\in I} p_i\right) =\hbox{weak$^*$-}
\sum\limits_{i\in I}T(p_i)
\end{equation} for every family $(p_i)_{i\in I}$ of mutually orthogonal
projections in $M.$
\end{enumerate}
\end{proposition}

\begin{proof} $(a)$ Let $(p_n)_{n\in \mathbb{N}}$ be a sequence of mutually
orthogonal projections in $M.$ Let us consider the element $z=\sum\limits_{n\in
I}\frac{\textstyle 1}{\textstyle n} p_n.$ By Lemma \ref{l linearity on single generated von Neumann subalgebras}
there exist skew-hermitian elements $a_{z},b_{z}\in M$ such that $T(x)=[a_{z},x] + b_{z}\circ x$ for
all $x\in \mathcal{W}^\ast(z).$ Since $\sum\limits_{n=1}^{\infty} p_n, p_m\in
\mathcal{W}^\ast(z),$ for all $m\in \mathbb{N},$ and the product of $M$ is jointly strong$^*$-continuous, we obtain that
$$
T \left(\sum\limits_{n=1}^{\infty} p_n\right)=\left[a_{z}, \sum\limits_{n=1}^{\infty} p_n\right] +
b_{z}\circ \left(\sum\limits_{n=1}^{\infty} p_n \right)$$ $$=\sum\limits_{n=1}^{\infty}[a_{z},
p_n] + \sum\limits_{n=1}^{\infty} b_{z}\circ p_n=\sum\limits_{n=1}^{\infty}  T(p_n),
$$ i.e. $T|_{\mathcal{P}(M)}$ is a countably or sequentially strong$^*$-additive mapping.\medskip

$(b)$ Let $\varphi$ be a positive normal functional in $M_*$, and let $\|.\|_{\varphi}$
 denote the prehilbertian seminorm given by $\|z\|_{\varphi}^2 = \frac12 \varphi (z z^* + z^* z)$
 ($z\in M$). Let $\{p_i\}_{i\in I}$ be an arbitrary family of mutually orthogonal projections in $M.$
 For every $n\in \mathbb{N}$ define $$I_n=\{i\in I: \left\|T (p_i)\right\|_{\varphi} \geq 1/n\}.$$

We claim, that $I_n$ is a finite set for every natural $n$. Otherwise, passing to a subset if necessary,
we can assume that there exists a natural $k$ such that
$I_k$ is infinite and countable. In this case the series
$\sum\limits_{i\in I_k} T(p_i) $ does not converge with respect to the semi-norm  $\|.\|_{\varphi}$.
On the other hand, since $I_k$ is a countable set, by $(a)$, we have
$$T \left(\sum\limits_{i\in I_k} p_i\right) = \hbox{strong$^*$-}\sum\limits_{i\in I_k} T(p_i), $$
which is impossible.  This proves the claim.\smallskip

We have shown that the set
$$ I_0=\left\{i\in I: \left\|T (p_i)\right\|_{\varphi} \neq 0\right\}=\bigcup\limits_{n\in \mathbb{N}}I_n $$
is a countable set, and $\left\|T (p_i)\right\|_{\varphi}= 0$, for every $i\in I\backslash I_0$.\smallskip

Set $p=\sum\limits_{i\in I\setminus I_0} p_i\in M.$ We shall show that $\varphi (T(p)) =0.$ Let $q$
 denote the support projection of $\varphi $ in $M$.
Having in mind that $\left\|T (p_i)\right\|^2_{\varphi}= 0$, for every $i\in I\backslash I_0$,
we deduce that $T(p_i) \perp q$ for every $i\in I\backslash I_0$.\smallskip

Replacing $T$ with $\widehat{T}=T - \delta (\frac12 T ({\mathbf 1}),{\mathbf 1})$  we can assume that $T({\mathbf 1}) = 0$
(cf. Lemma \ref{l 2-local T(1) skew}) and $T(x) = T(x)^*$, for every $x\in M_{sa}$ (cf. Lemma \ref{l sead}).
By Lemma \ref{l sea}, for every $i\in I\setminus I_0$ there exists a skew-hermitian element $a_i=a_{p,p_i}\in M$ such that
$$ T (p)=a_i p- p a_i,\hbox{ and, } T(p_i)=a_i p_i-p_i a_i. $$
Since $T(p_i) \perp q$ we get $(a_i p_i-p_i a_i) q = q (a_i p_i-p_i a_i) =0$, for all $i\in I\setminus I_0.$ Thus, since $p a_i p_i q= p_i a_i q$,
$$ (T(p) p_i)  q = (a_i p - p a_i) p_i q = a_i p_i q - p a_i p_i q  $$
$$= a_i p_i q - p_i a_i  q = (a_i p_i-p_i a_i) q =0,$$ and similarly
$$ q (p_i T(p)) = 0,$$ for every $i\in I\setminus I_0$. Consequently,
\begin{equation}\label{eq qpT(p)=0} (T(p) p) q=  T(p) \left(\sum\limits_{i\in I\setminus I_0}p_i \right) q = 0 =
q \left(\sum\limits_{i\in I\setminus I_0}p_i \right) T(p)=q (p T(p)).
\end{equation}
 Therefore, $$T(p) = \delta_{p,{\mathbf 1}} (p) = \delta_{p,{\mathbf 1}} \{p,p,p\}= 2 \{ \delta_{p,{\mathbf 1}}(p),p,p\} +
 \{p,\delta_{p,{\mathbf 1}} (p), p\}  $$ $$= 2 \{T(p),p,p\} +  \{p,T (p), p\} = p T(p) + T(p) p + p T(p)^* p $$
 $$=  p T(p) + T(p) p + p T(p) p,$$ which implies that $$\varphi (T(p)) = \varphi (p T(p) + T(p) p + p T(p) p) $$
 $$= \varphi (qp T(p) q) +\varphi(q T(p) p q) + \varphi (q p T(p) p q)= \hbox{(by \eqref{eq qpT(p)=0})}= 0.$$\smallskip

Finally, by $(a)$ we have
$$T\left(\sum\limits_{i\in I_0} p_i\right)= \|.\|_{\varphi}\hbox{-}\sum\limits_{i\in I_0} T\left(p_i\right).$$
Two more applications of $(a)$ give:
$$ \varphi \left(T\left(\sum\limits_{i\in I} p_i\right)\right)  = \varphi \left(
T\left(p +\sum\limits_{i\in I_0} p_i\right) \right)=
 \varphi \left(T(p) + T\left(\sum\limits_{i\in I_0} p_i\right)\right)$$
$$ = \varphi \left(T(p)\right) + \varphi \left( T\left(\sum\limits_{i\in I_0} p_i\right)\right) = \sum\limits_{i\in I_0} \varphi \left( T\left( p_i\right)\right).$$

By the Cauchy-Schwarz inequality, $0\leq \left|\varphi T(p_i)\right|^2 \leq \left\|T (p_i)\right\|^2_{\varphi}= 0$, for every $i\in I\backslash I_0$,
and hence $\displaystyle\sum\limits_{i\in I_0} \varphi \left( T\left( p_i\right)\right) = \sum\limits_{i\in I} \varphi \left( T\left( p_i\right)\right).$
The arbitrariness of $\varphi$ shows
that $T\left(\hbox{weak$^*$-}\sum\limits_{i\in I} p_i\right) =\hbox{weak$^*$-}
\sum\limits_{i\in I}T(p_i)$.
\end{proof}

Let $\phi$ be a normal functional in the predual of a von Neumann algebra $M.$
Our previous Proposition \ref{p AyupovKudaybergenov sigma complete additivity ternary}
assures that for every (not necessarily linear nor continuous) 2-local triple derivation
$T: M\to M$ the mapping $\phi\circ T|_{\mathcal{P}(M)}: \mathcal{P}(M) \to \mathbb{C}$ is
a completely additive mapping or a charge on $M$. Under the additional hypothesis of $M$ being a
 continuous von Neumann algebra or,  more generally, a von Neumann algebra with no
 Type I$_n$-factors ($1< n< \infty$) direct summands (i.e. without direct summand isomorphic
 to a matrix algebra $M_n(\mathbb{C})$, $1<n<\infty$), the Dorofeev--Sherstnev
 theorem (\cite[Theorem 29.5]{Shers2008} or \cite[Theorem 2]{Doro})
 imply that $\phi\circ T|_{\mathcal{P}(M)}$ is a bounded charge, that is,
 the set $\left\{ |\phi\circ T (p)| : p \in \mathcal{P}(M) \right\}$ is bounded. The uniform boundedness principle gives:

\begin{corollary}\label{c 2 local triple derivations are bounded charges} Let $M$ be a von Neumann algebra with no Type I$_n$-factor direct summands {\rm(}$1< n< \infty${\rm)} and let $T: M\to M$ be a {\rm(}not necessarily linear nor continuous{\rm)} 2-local triple derivation. Then the restriction $T|_{\mathcal{P}(M)}$ is a bounded weak$^*$-completely additive mapping.$\hfill\Box$
\end{corollary}

\subsection[Additivity of 2-local triple derivations
on hermitian parts]{Additivity of 2-local triple derivations
on hermitian parts of von Neumann algebras}\label{subsec:3.2}\ \medskip

Suppose now that $M$ is a von Neumann algebra with no Type I$_n$-factor direct
summands {\rm(}$1< n< \infty${\rm)}, and $T: M\to M$
is a {\rm(}not necessarily linear nor continuous{\rm)}
2-local triple derivation. By Corollary \ref{c 2 local triple derivations are bounded charges} combined
with the Bunce-Wright-Mackey-Gleason theorem \cite{BuWri92, BuWri94}, there
exits a bounded linear operator $G: M \to M$ satisfying that $G(p) = T(p)$, for every projection $p\in M$.\smallskip

Let $z$ be a self-adjoint element in $M$. By Lemma \ref{l
linearity on single generated von Neumann subalgebras}, there
exist skew-hermitian elements $a_{z},b_{z}\in M$ such that $T(x) =
[a_{z},x]+b_{z}\circ x,$ for every $x\in \mathcal{W}^\ast(z).$
Since $G|_{\mathcal{W}^\ast(z)}, T|_{\mathcal{W}^\ast(z)} :
\mathcal{W}^\ast(z) \to M$ are bounded linear operators, which
coincide on the set of projections of $\mathcal{W}^\ast(z)$, and
every self-adjoint element in $\mathcal{W}^\ast(z)$ can be
approximated in norm by finite linear combinations of mutually
orthogonal projections in $\mathcal{W}^\ast(z)$, it follows that
$T(x) = G(x)$ for every $x\in \mathcal{W}^\ast(z),$ and hence
$$
T(a) = G(a), \hbox{ for every } a\in M_{sa},
$$
in particular, $T$ is additive on $M_{sa}.$\smallskip

The above arguments materialize in the following result.

\begin{proposition}\label{p addit no type I_n} Let $T: M\to M$ be a {\rm(}not necessarily linear nor continuous{\rm)}
2-local triple derivation on a von Neumann algebra with no Type I$_n$-factor direct
summands {\rm(}$1< n< \infty${\rm)}. Then the restriction
$T|_{M_{sa}}$ is additive.$\hfill\Box$
\end{proposition}

\begin{corollary}\label{l:addit on properly infinite} Let $T: M\to M$ be a {\rm(}not necessarily linear nor continuous{\rm)} 2-local triple derivation on a
properly infinite von Neumann algebra. Then the restriction $T|_{M_{sa}}$ is additive.
\end{corollary}

Next we shall show that the conclusion of the above corollary is also true for a finite
von Neumann algebra.\smallskip

First we show that every 2-local triple derivation on a von Neumann algebra ``intertwines'' central projections.

\begin{lemma}\label{l:preserves central projections}
If $T$ is a {\rm(}not necessarily linear nor continuous{\rm)} $2$-local triple derivation on a von Neumann algebra
$M,$ and $p$ is a central projection in $M,$ then $T(Mp) \subset
Mp.$ In particular, $T(px) = pT(x)$ for every $x\in M$.
\end{lemma}

\begin{proof}
Consider $x \in Mp$, then $x = pxp = \{x,p,p\}$. $T$ coincides
with a triple derivation $\delta_{x,p}$ on the set $\{x,p\}$, hence
$$T(x)=\delta_{x,p} (x) = \delta_{x,p} \{ x,p, p\} = \{ \delta_{x,p} (x), p,p \} + \{x,
\delta_{x,p} (p), p \} + \{x,p,\delta_{x,p} (p)\}$$ lies in $Mp .$\smallskip

For the final statement, fix $x\in M,$ and consider skew-hermitian elements $a_{x,xp},$ $b_{x,xp}\in M$ satisfying
$$T(x) = [a_{x,xp},x]+ b_{x,xp}\circ x,\hbox{ and } T(xp) = [a_{x,xp},xp]+ b_{x,xp}\circ (xp).$$ The assumption $p$ being central implies that $pT(x) = T(px).\qedhere$
\end{proof}

\begin{proposition}\label{p addit finite} Let $T: M\to M$ be a {\rm(}not necessarily linear nor continuous{\rm)} 2-local triple derivation
on a finite von Neumann algebra.   Then the restriction $T|_{M_{sa}}$ is additive.
\end{proposition}

\begin{proof} Since $M$ is finite there exists a faithful normal semi-finite
trace $\tau$ on $M.$ We shall consider the following two cases.\smallskip

\emph{Case 1.} Suppose $\tau$ is a finite trace. Replacing $T$ with
$\widehat{T}=T - \delta (\frac12 T ({\mathbf 1}),{\mathbf 1})$ we can assume that
$T({\mathbf 1}) = 0$ (cf. Lemma \ref{l 2-local T(1) skew}) and $T(x) =
T(x)^*$, for every $x\in M_{sa}$ (cf. Lemma \ref{l sead}). By
Lemma \ref{l sea}, for every $x, y\in M_{sa}$ there exists a skew-hermitian element $a_{x,y}\in M$ such that $T(x)=[a_{x,y},x]$ and
$T (y)=[a_{x,y},y].$ Then
$$ T(x)y+xT(y)=[a_{x,y},x]y+x[a_{x,y},y]=[a_{x,y}, xy],$$
that is,
$$ [a_{x,y}, xy]=T(x) y+ xT(y).$$
Further
$$ 0 = \tau([a_{x,y}, xy])= \tau\left(T(x)y+xT(y)\right), $$
i.e. $\tau(T(x)y)=-\tau(x T(y))$, for every $x,y\in M_{sa}$. For arbitrary $u, v, w\in
M_{sa},$ set $x=u+v,$ and $y=w.$ The above identity implies
$$ \tau\left(T(u+v)w\right)=-\tau\left((u+v)T(w)\right)= $$
$$ =-\tau\left(uT(w)\right)- \tau\left(vT(w)\right)= \tau\left(T(u)w\right)+\tau\left(T(v)w\right)=\tau\left((T(u)+T(v))w\right), $$
and so
$$ \tau\left((T(u+v)-T(u)-T(v))w\right)=0$$ for all $u, v, w\in M_{sa}.$
Take $w=T(u+v)-T(u)-T(v).$ Then $\tau(ww^\ast)=0.$ Since the
trace $\tau$ is faithful it follows that  $ww^\ast=0,$ and hence
$w=0.$ Therefore
$$ T(u+v)=T(u)+T(v).$$\smallskip

\emph{Case 2.} As in \emph{Case 1}, we may assume $T({\mathbf 1})=0$. Suppose now that $\tau$ is a semi-finite trace. Since $M$ is finite there
exists a family of mutually orthogonal central projections
$\{z_i\}$ in $M$ such that $z_i$ has finite trace for all $i$ and
$\bigvee z_i=\mathbf{1}$ (cf. \cite[\S 2.2 or Corollary
2.4.7]{Sak}). By Lemma~\ref{l:preserves central
projections}, for each $i$, $T$ maps $z_i M$ into itself. From
Case 1, $T|_{z_i M}: z_i M \to z_i M$ is additive. Furthermore,
$$z_i T (x+y)= T|_{z_i M} (z_i x + z_i y) = T|_{z_i M} (z_i x)  +
T|_{z_i M} (z_i y) = z_i T(x) + z_i T(y),$$ for every $x,y\in M$
and every $i$. Therefore
$$ T(x+y) = \left(\sum_{i} z_i\right) T(x+y) = \sum_{i} z_i T(x+y) = \sum_{i} \left(z_i T(x) +z_i T(y)\right)
$$
$$= \left(\sum_{i} z_i\right) T(x) + \left(\sum_{i} z_i\right) T(y)  = T(x) + T(y), $$ for every $x,y\in
M.$ The proof is complete.
\end{proof}

Let $T: M\to M$ be a {\rm(}not necessarily linear nor continuous{\rm)} 2-local triple
derivation on an arbitrary von Neumann algebra. In this case there exist
orthogonal central projections $z_1, z_2\in M$ with
$z_1+z_2=\mathbf{1}$ such that: \begin{enumerate}[$(-)$]\item $z_1M$ is a finite von Neumann algebra;
\item $z_2M$ is a properly infinite von Neumann algebra,
\end{enumerate} (cf. \cite[\S 2.2]{Sak}).\smallskip

By Lemma~\ref{l:preserves central projections}, for each $k=1,2,$ $z_k T$ maps $z_k M$
into itself. By Corollary~\ref{l:addit on properly infinite}
and Proposition~\ref{p addit finite} both $z_1 T$ and $z_2 T$ are additive on $M_{sa}$. So
$T=z_1T+z_2T$ also is additive on $M_{sa}$.\smallskip

We have thus proved the following result:

\begin{proposition}\label{additgen} Let $T: M\to M$ be a {\rm(}not necessarily linear nor continuous{\rm)} 2-local triple
derivation on an arbitrary von Neumann algebra. Then the restriction $T|_{M_{sa}}$ is additive.$\hfill\Box$
\end{proposition}

\subsection{Main result}\label{subsec:3.3}\ \medskip

We can state now the main result of this paper.

\begin{theorem}\label{mainthm}
Let $M$ be  an arbitrary von Neumann algebra and let   $T: M\to M$
be a {\rm(}not necessarily linear nor continuous{\rm)} 2-local
triple derivation. Then $T$ is a triple derivation {\rm(}hence linear and continuous{\rm)}. Equivalently,
the set Der$_{t} (M)$, of all triple derivations on $M,$ is algebraically 2-reflexive in the set $\mathcal{M}(M)= M^M$
of all mappings from $M$ into $M$.
\end{theorem}

We need the following two Lemmata.

\begin{lemma}\label{jor} Let $T: M\to M$ be a {\rm(}not necessarily linear nor continuous{\rm)}
2-local triple derivation on a  von Neumann algebra with
$T(\mathbf{1})=0.$ Then there exists a skew-hermitian element
$a\in M$ such that $T(x)=[a, x],$ for all $x\in M_{sa}.$
\end{lemma}

\begin{proof}
Let $x\in M_{sa}.$ By Lemma~\ref{l sea} there exist a
skew-hermitian element $a_{x,x^2}\in M$ such that
$$
T(x)=[a_{x,x^2},x],\, T(x^2)=[a_{x,x^2}, x^2].
$$
Thus
$$
T (x^2)=[a_{x,x^2},x^2]=[a_{x,x^2},x]x+x[a_{x,x^2},x]=T(x)x+xT(x),
$$
i.e.
\begin{equation}\label{jord}
T (x^2)=T(x)x+xT(x),
\end{equation}
 for every $x\in M_{sa}$.\smallskip

By Proposition~\ref{additgen} and Lemma \ref{l sead}, $T|_{M_{sa}} : M_{sa}
\to M_{sa}$ is a real linear mapping. Now, we consider the linear extension
$\hat{T}$ of $T|_{M_{sa}}$ to $M$ defined by
$$ \hat{T}(x_1+ix_2)=T(x_1)+i T(x_2),\, x_1, x_2\in M_{sa}.$$

Taking into account the homogeneity of $T,$ Proposition~\ref{additgen}
and the identity~\eqref{jord} we obtain that $\hat{T}$ is a
Jordan derivation on $M.$  By \cite[Theorem 1]{Bre} any Jordan
derivation on a semi-prime algebra is a derivation. Since $M$ is
von Neumann algebra, $\hat{T}$ is a derivation on $M$ (see also \cite{Sinclair70} and \cite{John96}).
Therefore there exists an element $a\in M$ such that $\hat{T}(x)=[a,x]$
for all $x\in M.$ In particular, $T(x)=[a, x]$ for all $x\in
M_{sa}.$ Since $T(M_{sa}) \subseteq M_{sa}$, we can assume that
$a^\ast=-a$, which completes the proof. \end{proof}

\begin{lemma}\label{l:vanish on self-adjoint} Let $T: M\to M$ be a {\rm(}not necessarily linear nor continuous{\rm)}
2-local triple derivation on a  von Neumann algebra. If $T|_{M_{sa}}\equiv 0,$ then
$T\equiv 0.$
\end{lemma}

\begin{proof}
Let $x\in M$ be an arbitrary element and let $x=x_1+ix_2,$ where
$x_1, x_2\in M_{sa}.$ Since $T$ is homogeneous, if necessary,
passing to the element $(1+\|x_2\|)^{-1} x,$ we can suppose that
$\|x_2\|<1.$ In this case the element $y=\mathbf{1}+x_2$ is
positive and invertible. Take skew-hermitian elements $a_{x,y}, b_{x,y}\in M$
such that
$$
T(x)=[a_{x,y},x]+b_{x,y}\circ x,
$$
$$
T(y)=[a_{x,y},y]+b_{x,y}\circ y.
$$
Since $T(y)=0,$ we get $[a_{x,y},y]+b_{x,y}\circ y=0.$ By Lemma~\ref{l:ax=xa,
bx=xb=0} we obtain that $[a_{x,y}, y]=0$ and $ib_{x,y} \circ y=0.$ Taking into
account that $ib_{x,y}$ is hermitian, $y$ is positive and invertible,
Lemma~\ref{l:product} implies that $b_{x,y}=0.$\smallskip

We further note that
$$
0=[a_{x,y},y]=[a_{x,y}, \mathbf{1}+x_2]=[a_{x,y}, x_2],
$$
i.e.
$$
[a_{x,y}, x_2]=0.
$$
Now,
$$
T(x)=[a_{x,y},x]+b_{x,y}\circ x=[a_{x,y}, x_1+ix_2]=[a_{x,y}, x_1],
$$
i.e.
$$
T(x)=[a_{x,y}, x_1].
$$
Therefore,
$$
T(x)^\ast=[a_{x,y}, x_1]^\ast=[x_1, a_{x,y}^\ast]=[x_1, -a_{x,y}]=[a_{x,y}, x_1]=T(x).
$$
So
\begin{equation}\label{iii} T(x)^\ast=T(x).
\end{equation}

Now replacing $x$ by $ix $ on~\eqref{iii} we obtain from the homogeneity of $T$ that
\begin{equation}\label{iv}T(x)^\ast=-T(x).
\end{equation}
Combining \eqref{iii} and \eqref{iv} we obtain that $T(x)=0,$ which finishes the proof. \end{proof}

\begin{proof}[\textit{Proof of Theorem~\ref{mainthm}}] Let us define $\widehat{T}=T - \delta \left(\frac12 T ({\mathbf 1}),{\mathbf 1}\right).$
Then $\widehat{T}$ is a 2-local triple derivation on $M$
with $\widehat{T}({\mathbf 1}) = 0$ (cf. Lemma \ref{l 2-local T(1)
skew}) and $\widehat{T}(x) = \widehat{T}(x)^\ast,$ for every $x\in M_{sa}$ (cf. Lemma
\ref{l sead}). By Lemma~\ref{jor} there exists an element $a\in M$
such that $\widehat{T}(x)=[a,x]$ for all $x\in M_{sa}.$ Consider the
$2$-local triple derivation $\widehat{T}-[a,\cdot].$ Since $(\widehat{T}-[a,
\cdot])|_{M_{sa}}\equiv 0,$ Lemma~\ref{l:vanish on self-adjoint}
implies that $\widehat{T}=[a, \cdot],$ and hence $T = [a, \cdot] + \delta \left(\frac12 T ({\mathbf 1}),{\mathbf 1}\right),$ witnessing the desired statement.
\end{proof}

\end{document}